\documentclass[a4paper]{amsart}

\usepackage{amsfonts}
\usepackage{amssymb}
\usepackage{amsthm}
\usepackage{amsmath}
\usepackage[all]{xy}

\SelectTips{cm}{}

\theoremstyle{plain}
\newtheorem{theorem}{Theorem}[section]
\newtheorem*{theorem-non}{Theorem}
\newtheorem{proposition}[theorem]{Proposition}
\newtheorem*{proposition-non}{Proposition}
\newtheorem{corollary}[theorem]{Corollary}

\newtheorem*{conjecture-non}{Conjecture}
\newtheorem{lemma}[theorem]{Lemma}
\newtheorem*{lemma-non}{Lemma}

\theoremstyle{definition}
\newtheorem{definition}[theorem]{Definition}

\theoremstyle{remark}
\newtheorem{remark}[theorem]{Remark}
\newtheorem{example}[theorem]{Example}
\newtheorem*{example-non}{Example}

\numberwithin{equation}{section}

\DeclareMathOperator{\diag}{diag}

\DeclareMathOperator{\ch}{ch}
\DeclareMathOperator{\Idem}{Idem}

\newcommand{\id}{\textup{id}}
\newcommand{\pr}{\textup{pr}}
\newcommand{\map}{\textup{map}}

\newcommand{\trans}{\textup{trans}}
\newcommand{\hfd}{\textup{hfd}}
\newcommand{\hf}{\textup{hf}}
\newcommand{\f}{\textup{f}}
\newcommand{\pt}{\textup{pt}}
\newcommand{\inc}{\textup{inc}}
\newcommand{\sing}{\textup{sing}}

\def\N{\mathbb N}
\def\R{\mathbb R}
\def\Z{\mathbb Z}

\title{The $K$-theoretic Farrell-Jones Conjecture for CAT(0)-groups}
\author{Christian Wegner}
\address{Mathematisches Institut \\ Universit\"at Bonn \\ Endenicher Allee 60 \\ Bonn, D-53115 \\ Germany}
\email{wegner@math.uni-bonn.de}

\begin{document}

\begin{abstract}
We prove the $K$-theoretic Farrell-Jones conjecture with (twisted) coefficients for CAT(0)-groups.
\end{abstract}

\maketitle

\section{Introduction}

The $K$-theoretic Farrell-Jones conjecture with coefficients for a group $G$ says that the $K$-theoretic assembly map
\[
 H^G_m(E_{\mathcal{VC}yc}G;\mathbf{K}_\mathcal{A}) \to H^G_m(\pt;\mathbf{K}_\mathcal{A}) \cong K_m(\int_G \mathcal{A})
\]
is an isomorphism for all $m \in \Z$ and every additive category $\mathcal{A}$ with a strict right $G$-action. Here $E_{\mathcal{VC}yc}G$ denotes the classifying space of the group $G$ with respect to the family of virtually cyclic subgroups. Any additive category $\mathcal{A}$ with a right $G$-action induces a covariant functor $\mathbf{K}_\mathcal{A}$ from the orbit category of $G$ to the category of spectra with (strict) maps of spectra as morphisms (see \cite[Definition 3.1]{BR07}). We denote the associated $G$-homology theory by $H^G_*(-;\mathbf{K}_\mathcal{A})$ (see \cite[sections 4 and 7]{DL98}). The assembly map is the map induced by the projection $E_{\mathcal{VC}yc}G \to \pt$ onto the space consisting of one point.

The $K$- and $L$-theoretic Farrell-Jones conjecture plays an important role in the classification and geometry of manifolds. Moreover, it implies a variety of well-known conjectures, e.g. the Bass-, Borel-, Kaplansky- and Novikov-conjecture. For more information on the Farrell-Jones conjecture we refer to the survey article \cite{LR05}.

In this paper we prove the $K$-theoretic Farrell-Jones conjecture with coefficients for CAT(0)-groups. By a CAT(0)-group we mean a group which admits a cocompact proper action by isometries on a finite dimensional CAT(0)-space.
The proof is based on methods from \cite{BLR08}, \cite{BL09} and \cite{BL10}. In \cite{BLR08} Bartels, L\"uck and Reich show the bijectivity of the $K$-theoretic assembly map for hyperbolic groups. In \cite{BL09} and \cite{BL10} Bartels and L\"uck investigate the $K$-theoretic assembly map for CAT(0)-groups and prove bijectivity in degree $m \leq 0$ and surjectivity in degree $m = 1$.

The general strategy to prove the $K$-theoretic Farrell-Jones conjecture is to study the obstruction category $\mathcal{O}^G(E_\mathcal{F}G,\pt;\mathcal{A})$ whose $K$-theory gives the homotopy fiber of the $K$-theoretic assembly map. Then a transfer map has to be constructed which allows to replace the one-point-space by a suitable metric space which gives room for certain constructions. This metric space has to be carefully chosen since we need contracting properties afterwards. In the case of hyperbolic groups this space is a compactification of the Rips complex of the group $G$. In the case of CAT(0)-groups we use large closed balls in the associated CAT(0)-space. Finally, contractible maps on the metric spaces are used to gain control and to prove the vanishing of the $K$-theory groups of the obstruction category.

The main difficulty in enlarging the result of Bartels and L\"uck comes from the fact that the closed balls in the CAT(0)-space are no $G$-spaces. They only admit a homotopy $G$-action. This is sufficient to define the transfer map for $K_1$ since this map only requires homotopy chain actions. But for higher $K$-theory we have to take account of higher homotopies. A useful tool to tackle this problem are strong homotopy actions which we introduce in section~\ref{sec-cha}. They describe in a simple way a homotopy action together with all higher homotopies. We use them in section~\ref{sec-str} where we define the notion of strong transfer reducibility for groups. This definition specifies the requirements that we have on the metric space which replaces the one-point-space. We show that hyperbolic groups and CAT(0)-groups are strongly transfer reducible over the family of virtually cyclic subgroups (see Example~\ref{ex-hyp} and Theorem~\ref{thm-cat0}).

The following sections are dedicated to the proof of the $K$-theoretic Farrell-Jones conjecture with coefficients for groups which are strongly transfer reducible. More precisely, we prove
\begin{theorem} \label{thm-main}
Let $G$ be a group which is strongly transfer reducible over a family $\mathcal{F}$ of subgroups of $G$.
Let $\mathcal{A}$ be an additive $G$-category, i.e. an additive category with a strict right $G$-action by functors of additive categories.
Then the $K$-theoretic assembly map
\begin{equation} \label{eq-assembly}
 H^G_m(E_\mathcal{F}G;\mathbf{K}_\mathcal{A}) \to H^G_m(\pt;\mathbf{K}_\mathcal{A}) \cong K_m(\int_G \mathcal{A})
\end{equation}
is an isomorphism for all $m \in \Z$.
\end{theorem}

In section~\ref{sec-obscat} we give a short review of controlled algebra which is a crucial tool in the proof. In particular, we define the obstruction category. An outline of the proof of Theorem~\ref{thm-main} is given in section~\ref{sec-out}. The last two sections deal with the transfer map and finish the proof of Theorem~\ref{thm-main}.

Following the proof of \cite[Lemma 2.3]{BL10} we see that Theorem~\ref{thm-main} and Theorem~\ref{thm-cat0} imply
\begin{corollary}
Let $G_1$, $G_2$ be groups which satisfy the $K$-theoretic Farrell-Jones conjecture with coefficients.
Then the groups $G_1 \times G_2$ and $G_1 * G_2$ satisfy the $K$-theoretic Farrell-Jones conjecture with coefficients, too.
\end{corollary}

This paper was supported by the SFB 878 -- Groups, Geometry \& Actions.

\section{Strong homotopy actions} \label{sec-cha}

Let $G$ be a CAT(0)-group, i.e. a group which admits a cocompact proper action by isometries on a finite dimensional CAT(0)-space $Y$. We would like to replace the CAT(0)-space $Y$ by a compact space, namely a large ball in $Y$. The price we have to pay for this replacement is that we only retain a $G$-action on the ball up to homotopy. To control these homotopies we introduce the notion of a strong homotopy action.

\begin{definition} \label{def-cha}
A \emph{strong homotopy action} of a group $G$ on a topological space $X$ is a continuous map
\[
 \Psi \colon \coprod_{j=0}^\infty \big( (G \times [0,1])^j \times G \times X \big) \to X
\]
with the following properties:
\begin{enumerate}
 \item $\Psi(\ldots,g_l,0,g_{l-1},\ldots) = \Psi(\ldots,g_l,\Psi(g_{l-1},\ldots))$
 \item $\Psi(\ldots,g_l,1,g_{l-1},\ldots) = \Psi(\ldots,g_l \cdot g_{l-1},\ldots)$
 \item $\Psi(e,t_j,g_{j-1},\ldots) = \Psi(g_{j-1},\ldots)$
 \item $\Psi(\ldots,t_l,e,t_{l-1},\ldots) = \Psi(\ldots,t_l \cdot t_{l-1},\ldots)$
 \item $\Psi(\ldots,t_1,e,x) = \Psi(\ldots,x)$
 \item $\Psi(e,x) = x$
\end{enumerate}
\end{definition}
A strong homotopy action restricts to a homotopy $S$-action in the sense of \cite[Definition 1.4]{BL09} by setting $\phi_g(x) := \Psi(g,x)$ and $H_{g,h}(x,t) := \Psi(g,t,h,x)$. But the strong homotopy action also perceives the higher homotopies. The simple description of a strong homotopy action is very useful for our purpose.

\begin{remark} \label{rem_cha}
It is not true in general that a topological space $X$ which is homotopy equivalent to a $G$-space admits a strong homotopy action by conjugation with the homotopy equivalence.\\
Strong homotopy actions appear in the following situation: Let $Y$ be a $G$-space and let $H \colon Y \times [0,1] \to Y$ be a deformation retraction onto a subspace $X \subseteq Y$ (i.e. $H_0(Y) = X$, $H_0|_X = \id_X$ and $H_1 = \id_Y$) such that $H_t \circ H_{t'} = H_{t \cdot t'}$ for all $t,t' \in [0,1]$. For example, we can consider the CAT(0)-space together with a deformation retraction on a ball by projecting along geodesics. (We will make this more precise in the proof of Theorem~\ref{thm-cat0}.) In this situation we define
\[
 \Omega \colon \coprod_{j=0}^\infty \big( (G \times [0,1])^j \times G \times X \big) \to Y
\]
inductively by $\Omega(g_0,x) := g_0 \cdot x$ and $\Omega(g_j,t_j,g_{j-1},\dots) := g_j \cdot H_{t_j}(\Omega(g_{j-1},\ldots))$ for $j \geq 1$. Then $\Psi := H_0 \circ \Omega$ is a strong homotopy action.\\
On the other hand, a strong homotopy action $\Psi$ induces a subspace $M$ of the space of continuous mappings $\coprod_{j=0}^\infty (G \times [0,1])^j \times G \to X$. It is defined by
\[
 M := \big\{ \Psi(?,t_\alpha,g_{\alpha-1},\ldots,g_0,x) \, \big| \, \alpha \in \N_0, t_i \in [0,1], g_i \in G, x \in X \big\}
\]
and has a $G$-action given by $c_g(f) := f(?,1,g)$. Moreover, we obtain a deformation retraction $H \colon M \times [0,1] \to M, f \mapsto f(?,t,e)$ onto the subspace $H_0(M) = \{\Psi(?,x) \mid x \in X\} \cong X$ which satisfies $H_t \circ H_{t'} = H_{t \cdot t'}$ for all $t,t' \in [0,1]$.\\
If we start with a strong homotopy action and construct the associated deformation retraction then the strong homotopy action associated to this deformation retraction coincides with the original strong homotopy action. In general, the other composition of the two constructions is not the identity. Nevertheless, both constructions are inverse to each other if the interior $\mathring{X}$ of $X = H_0(Y)$ satisfies $G \cdot \mathring{X} = Y$. This condition is satisfied in the case of our CAT(0)-group as long as the ball, on which we project, is large enough. Anyhow, we will not make use of this fact.
\end{remark}

In analogy to \cite[Definition 1.4 and Definition 3.4]{BL09} we make the following definition.
\begin{definition}
Let $\Psi$ be a strong homotopy $G$-action on a metric space $(X,d_X)$. Let $S \subseteq G$ be a finite symmetric subset which contains the trivial element $e \in G$. Let $k \in \N$ be a natural number.
\begin{enumerate}
 \item For $g \in G$ we define $F_g(\Psi,S,k) \subset \map(X,X)$ by
       \[
        F_g(\Psi,S,k) := \big\{ \Psi(g_k,t_k,\ldots,g_0,?) \colon X \to X \, \big| \, g_i \in S, t_i \in [0,1], g_k \cdot \ldots \cdot g_0 = g \big\}.
       \]
 \item For $(g,x) \in G \times X$ we define $S^1_{\Psi,S,k}(g,x) \subset G \times X$ as the subset consisting of all $(h,y) \in G \times X$ with the following property: There are $a,b \in S$, $f \in F_a(\Psi,S,k)$ and $\tilde{f} \in F_b(\Psi,S,k)$ such that $f(x)=\tilde{f}(y)$ and $h = g a^{-1} b$. For $n \in \N^{\geq 2}$ we set
       \[
        S^n_{\Psi,S,k}(g,x) := \big\{ S^1_{\Psi,S,k}(h,y) \, \big| \, (h,y) \in S^{n-1}_{\Psi,S,k}(g,x) \big\}.
       \]
 \item For $\Lambda \in \R^{>0}$ we define the quasi-metric $d_{\Psi,S,k,\Lambda}$ on $G \times X$ as the largest quasi-metric on $G \times X$ satisfying
     \begin{itemize}
      \item $d_{\Psi,S,k,\Lambda}\big((g,x),(g,y)\big) \leq \Lambda \cdot d_X(x,y)$ for all $g \in G$, $x,y \in X$ and
      \item $d_{\Psi,S,k,\Lambda}\big((g,x),(h,y)\big) \leq 1$ for all $(h,y) \in S^1_{\Psi,S,k}(g,x)$.
     \end{itemize}
\end{enumerate}
\end{definition}
We remind the reader that the difference between a metric and a quasi-metric is that in the later case the distance $\infty$ is allowed. Notice that the quasi-metric $d_{\Psi,S,k,\Lambda}$ is $G$-invariant with respect to the $G$-action $g(h,x) := (gh,x)$ on $G \times X$. See \cite[Definition 3.4]{BL09} for a construction of the quasi-metric $d_{\Psi,S,k,\Lambda}$. The following lemma is taken from \cite[Lemma 3.5]{BL09}.
\begin{lemma} \label{lem-qm}
\begin{enumerate}
\item The subset $S$ generates $G$ if and only if $d_{\Psi,S,k,\Lambda}$ is a metric.
\item Let $(g,x),(h,y) \in G \times X$ and $n \in \N$. Then $(h,y) \in S^n_{\Psi,S,k}(g,x)$ if and only if $d_{\Psi,S,k,\Lambda}((g,x),(h,y)) \leq n$ for all $\Lambda > 0$. \label{lem-qm2}
\item The topology on $G \times X$ induced by $d_{\Psi,S,k,\Lambda}$ coincides with the product topology.
\end{enumerate}
\end{lemma}

\section{Strong transfer reducibility} \label{sec-str}

In this section we introduce the notion of strong transfer reducibility for groups which is an analogue of the notion of transfer reducibility defined in \cite[Definition 1.8]{BL09}.

\begin{definition} \label{def-str}
Let $\mathcal{F}$ be a family of subgroups of $G$.
The group $G$ is called \emph{strongly transfer reducible} over $\mathcal{F}$ if there exists a natural number $N \in \N$ with the following property:
For every finite symmetric subset $S \subseteq G$ containing the trivial element $e \in G$ and every natural numbers $k, n \in \N$ there are
\begin{itemize}
 \item a compact contractible controlled $N$-dominated metric space $X$,
 \item a strong homotopy $G$-action $\Psi$ on $X$ and
 \item a cover $\mathcal{U}$ of $G \times X$ by open sets
\end{itemize}
such that
\begin{enumerate}
 \item $\mathcal{U}$ is an open $\mathcal{F}$-cover,
 \item $\dim(\mathcal{U}) \leq N$,
 \item for every $(g,x) \in G \times X$ there exists $U \in \mathcal{U}$ with $S^n_{\Psi,S,k}(g,x) \subseteq U$.
\end{enumerate}
\end{definition}

\begin{example} \label{ex-hyp}
Hyperbolic groups are strongly transfer reducible over the family of virtually cyclic subgroups. As metric space $X$ we choose the compactification of the Rips complex. The strong homotopy action $\Psi$ on $X$ is given by the action of the hyperbolic group on $X$: $\Psi(g_j,t_j,\ldots,g_0,x) := g_j \cdot \ldots \cdot g_0 \cdot x$. For more details we refer to the proof of \cite[Proposition 2.1]{BL09}.
\end{example}

\begin{remark} \label{rem-subgr}
Let $G$ be a group which is strongly transfer reducible over a family of subgroups $\mathcal{F}$. Let $H < G$ be a subgroup. We set $\mathcal{F}_H := \{ F \cap H \mid F \in \mathcal{F} \}$. By restricting the strong homotopy $G$-action $\Psi$ on $X$ we obtain a strong homotopy $H$-action. Moreover, we can restrict the cover $\mathcal{U}$ of $G \times X$ to a cover of $H \times X$. We conclude that $H$ is strongly transfer reducible over $\mathcal{F}_H$.
\end{remark}

Following \cite{BL10} we obtain
\begin{theorem} \label{thm-cat0}
Every CAT(0)-group is strongly transfer reducible over the family of virtually cyclic subgroups.
\end{theorem}
\begin{proof}
Let $Y$ be a finite dimensional CAT(0)-space on which $G$ acts cocompactly and properly. Fix a base point $x_0 \in Y$. As metric space $X$ we will choose a (large) closed ball $\overline{B}_R(x_0) \subseteq Y$ around the base point. By \cite[Lemma 6.2]{BL10} $\overline{B}_R(x_0)$ is a compact contractible controlled $(2 \cdot \dim(Y) + 1)$-dominated metric space.

For $R > 0$ we define a strong homotopy action
\[
 \Psi^R \colon \coprod_{j=0}^\infty \big( (G \times [0,1])^j \times G \times \overline{B}_R(x_0) \big) \to \overline{B}_R(x_0)
\]
as follows.
A continuous map $c \colon \R \to Y$ is called a generalized geodesic if there are $c_-, c_+ \in \R \cup \{\pm \infty\}$ with $\infty \neq c_- \leq c_+ \neq -\infty$ such that $c$ restricts to an isometry on the interval $(c_-,c_+)$ and is locally constant on the complement of this interval (see \cite[Definition 1.1]{BL10}).
For $x,y \in Y$ we denote by $c_{x,y}$ the generalized geodesic satisfying $(c_{x,y})_- = 0$, $c_{x,y}(-\infty) = x$ and $c_{x,y}(\infty) = y$. For $R > 0$ we consider the deformation retraction $H^R \colon Y \times [0,1] \to Y$ on the ball $\overline{B}_R(x_0)$ by projecting along geodesics, i.e.
\[
 H^R(x,t) := c_{x,x_0}\big((d_Y(x,x_0)-R) \cdot (1-t)\big).
\]
Notice that $H^R_t \circ H^R_{t'} = H^R_{t \cdot t'}$. We define $\Psi^R$ as the associated strong homotopy action (see Remark~\ref{rem_cha}).

For the construction of the cover $\mathcal{U}$ we have to introduce the flow space $FS(Y)$ which is the $G$-space consisting of all generalized geodesics $c \colon \R \to Y$. The metric on $FS(Y)$ is given by
\[
 d_{FS(Y)}(c,d) := \int_{-\infty}^\infty \frac{d_Y(c(t),d(t))}{2 \cdot e^{|t|}} \, dt.
\]
We define a $G$-equivariant flow $\Phi \colon FS(Y) \times \R \to FS(Y)$ by $\Phi_\tau(c)(t) := c(t + \tau)$ (see \cite[Definition 1.2]{BL10}).

By \cite[Theorem 5.7]{BL10} and \cite[subsection 6.3]{BL10} there is $\widehat{N} \in \N$ such that for every $\alpha > 0$ there exists an open $\mathcal{VC}yc$-cover $\mathcal{V}$ of $FS(Y)$ of dimension at most $\widehat{N}$ and $\epsilon > 0$ such that
\begin{itemize}
 \item $\mathcal{V} / G$ is finite and
 \item for every $z \in FS(Y)$ there is $V \in \mathcal{V}$ with $B_\epsilon(\Phi_{[-\alpha,\alpha]}(z)) \subseteq V$.
\end{itemize}
We set $N:= \max\{\widehat{N}, 2 \cdot \dim(Y)+1\}$.

The construction of the cover $\mathcal{U}$ is based on the contracting property described in Lemma~\ref{lem-cat0} below. We fix $\alpha > 0$ as in the assertion of Lemma~\ref{lem-cat0}. Let $\mathcal{V}$ be an open $\mathcal{VC}yc$-cover of $FS(Y)$ of dimension at most $\widehat{N}$ and let $\epsilon$ be a positive real number such that the two properties mentioned above are satisfied.
For this $\epsilon > 0$ we obtain $R,T > 0$ from Lemma~\ref{lem-cat0}. Then the cover
\[
 \mathcal{U} := \big\{ (\Phi_T \circ \iota)^{-1}(V) \cap G \times \overline{B}_R(x_0) \, \big| \, V \in \mathcal{V} \big\}.
\]
with $\iota \colon G \times \overline{B}_R(x_0) \to FS(Y), (g,y) \mapsto c_{gx_0,gy}$ has the desired properties.
\end{proof}

In the proof of Theorem~\ref{thm-cat0} we used the following lemma which is a modification of \cite[Proposition 3.8]{BL10} resp. \cite[Lemma 5.12]{BL10}.
\begin{lemma} \label{lem-cat0}
Let $S \subseteq G$ be a finite symmetric subset containing the trivial element $e \in G$. Let $k, n \in \N$. Then there exists $\alpha > 0$ with the following property:
For all $\epsilon > 0$ there are $R,T > 0$ such that for every $(g,x) \in G \times \overline{B}_R(x_0)$ and $(h,y) \in S^n_{\Psi^R,S,k}(g,x)$ there is $\tau \in [-\alpha,\alpha]$ with
\[
 d_{FS(Y)}\big(\Phi_T \circ \iota(g,x),\Phi_{T+\tau} \circ \iota(h,y)\big) \leq \epsilon.
\]
(We use the same notation as in the proof of Theorem~\ref{thm-cat0}.)
\end{lemma}
\begin{proof}
We set $\alpha := 2n \cdot (k+1) \cdot \alpha'$ with $\alpha' := \max\{ d_Y(g x_0, h x_0) \, | \, g,h \in S^{k+1} \}$. Let $\epsilon > 0$.
By \cite[Proposition 3.5]{BL10} there exist $R,T > 0$ such that for $x, x', x'' \in Y$ with $d_Y(x',x'') \leq \alpha'$ and $x \in \overline{B}_{R+\alpha}(x')$ there is $\tau \in [-\alpha',\alpha']$ such that
\begin{equation} \label{eq1}
 d_{FS(Y)}\big(\Phi_T(c_{x',c_{x',x}(d_Y(x',x)-R)}),\Phi_{T+\tau}(c_{x'',c_{x'',x}(d_Y(x'',x)-R)})\big) \leq \frac{\epsilon}{2n \cdot (k+1) \cdot e^{2\alpha}}.
\end{equation}
We fix such positive real numbers $R,T$.

At next we show that for every $z \in \overline{B}_R(x_0)$, $a \in G$ and $f \in F_a(\Psi^R,S,k)$ there is $\tau \in [-(k+1) \cdot \alpha',(k+1) \cdot \alpha']$ satisfying
\begin{equation} \label{eq2}
 d_{FS(Y)}\big(\Phi_T \circ \iota(e,z),\Phi_{T+\tau} \circ \iota(a^{-1},f(z))\big) \leq \frac{\epsilon}{2n \cdot e^\alpha}.
\end{equation}
There are $g_0, \ldots, g_k \in S$ and $t_1, \ldots, t_k \in [0,1]$ such that $f(z) = \Psi^R(g_k,t_k,\ldots,g_0,z)$.
We set $y_0 := z$ and $z_0 := z$. Moreover, we define $x_i := (g_{i-1} \ldots g_0)^{-1} \cdot x_0$, $y_i := (g_{i-1} \ldots g_0)^{-1} \cdot \Psi^R(g_{i-1},\ldots,g_0,z)$ and $z_i := (g_{i-1} \ldots g_0)^{-1} \cdot \Omega^R(g_{i-1},\ldots,g_0,z)$ for $i=1,\ldots,k+1$.
Notice that $y_i$ and $z_{i+1}$ lie on the geodesic between $x_i$ and $z_i$. We conclude $y_i = c_{x_i,z_{i+1}}(d_Y(x_i,z_{i+1})-R)$ and $y_{i+1} = c_{x_{i+1},z_{i+1}}(d_Y(x_{i+1},z_{i+1})-R)$.
We apply the inequality~(\ref{eq1}) and obtain $\tau_i \in [-\alpha',\alpha']$ ($i=0,\ldots,k$) with
\[
 d_{FS(Y)}\big(\Phi_T(c_{x_i,y_i}),\Phi_{T+\tau_i}(c_{x_{i+1},y_{i+1}})\big) \leq \frac{\epsilon}{2n \cdot (k+1) \cdot e^{2\alpha}}.
\]
We set $s_i := \sum_{l=0}^{i-1} \tau_l \leq (k+1) \cdot \alpha'$ for $i=0,\ldots,k+1$.
Using \cite[Lemma 1.3]{BL10} we calculate
\begin{align*}
& d_{FS(Y)}\big(\Phi_T \circ \iota(e,z),\Phi_{T+s_{k+1}} \circ \iota(a^{-1},f(z))\big) \\
& = d_{FS(Y)}\big(\Phi_{T+s_0}(c_{x_0,y_0}),\Phi_{T+s_{k+1}}(c_{x_{k+1},y_{k+1}})\big) \\
& \leq \sum_{i=0}^k d_{FS(Y)}\big(\Phi_{T+s_i}(c_{x_i,y_i}),\Phi_{T+s_{i+1}}(c_{x_{i+1},y_{i+1}})\big) \\
& \leq \sum_{i=0}^k e^{|s_i|} \cdot d_{FS(Y)}\big(\Phi_{T}(c_{x_i,y_i}),\Phi_{T+\tau_i}(c_{x_{i+1},y_{i+1}})\big) \\
& \leq \sum_{i=0}^k e^\alpha \cdot \frac{\epsilon}{2n \cdot (k+1) \cdot e^{2\alpha}} = \frac{\epsilon}{2n \cdot e^\alpha}.
\end{align*}

Since $(h,y) \in S^n_{\Psi^R,S,k}(g,x)$, there are $c_i \in \overline{B}_R(x_0)$, $a_i,b_i \in S$, $f_i \in F_{a_i}(\Psi^R,S,k)$ and $\tilde{f}_i \in F_{b_i}(\Psi^R,S,k)$ ($i=1,\ldots,n)$ such that $c_0=x$, $c_n=y$, $f_i(c_{i-1})=\tilde{f}_i(c_i)$ and $h = g a_1^{-1} b_1 \ldots a_n^{-1} b_n$.
By the inequality~(\ref{eq2}) there are $\tau_i, \tilde{\tau}_i \in [-(k+1) \alpha',(k+1) \alpha']$ such that
\begin{eqnarray*}
 d_{FS(Y)}\big(\Phi_T \circ \iota(e,c_{i-1}),\Phi_{T+\tau_i} \circ \iota(a_i^{-1},f_i(c_{i-1}))\big) & \leq & \frac{\epsilon}{2n \cdot e^\alpha}, \\
 d_{FS(Y)}\big(\Phi_T \circ \iota(e,c_i),\Phi_{T+\tilde{\tau}_i} \circ \iota(b_i^{-1},\tilde{f}_i(c_i))\big) & \leq & \frac{\epsilon}{2n  \cdot e^\alpha}
\end{eqnarray*}
for $i=1,\ldots,n$.
We set $g_i := g a_1^{-1} b_1 \ldots a_i^{-1} b_i$, $\sigma_i := \sum_{l=1}^i \tau_i - \tilde{\tau}_i \in [-\alpha,\alpha]$.
Since $g_{i-1}a_i^{-1} = g_i b_i^{-1}$, $f_i(c_{i-1})=\tilde{f}_i(c_i)$ and $T+\sigma_{i-1}+\tau_i=T+\sigma_i+\tilde{\tau}_i$ we conclude
\begin{align*}
& d_{FS(Y)}\big(\Phi_{T+\sigma_{i-1}} \circ \iota(g_{i-1},c_{i-1}),\Phi_{T+\sigma_i} \circ \iota(g_i,c_i)\big) \\
& \leq d_{FS(Y)}\big(\Phi_{T+\sigma_{i-1}} \circ \iota(g_{i-1},c_{i-1}),\Phi_{T+\sigma_{i-1}+\tau_i} \circ \iota(g_{i-1}a_i^{-1},f_i(c_{i-1}))\big) + \\
& \qquad d_{FS(Y)}\big(\Phi_{T+\sigma_i+\tilde{\tau}_i} \circ \iota(g_i b_i^{-1},\tilde{f}_i(c_i)),\Phi_{T+\sigma_i} \circ \iota(g_i,c_i)\big) \\
& = d_{FS(Y)}\big(\Phi_{T+\sigma_{i-1}} \circ \iota(e,c_{i-1}),\Phi_{T+\sigma_{i-1}+\tau_i} \circ \iota(a_i^{-1},f_i(c_{i-1}))\big) + \\
& \qquad d_{FS(Y)}\big(\Phi_{T+\sigma_i+\tilde{\tau}_i} \circ \iota(b_i^{-1},\tilde{f}_i(c_i)),\Phi_{T+\sigma_i} \circ \iota(e,c_i)\big) \\
& \leq e^{|\sigma_{i-1}|} \cdot d_{FS(Y)}\big(\Phi_T \circ \iota(e,c_{i-1}),\Phi_{T+\tau_i} \circ \iota(a_i^{-1},f_i(c_{i-1}))\big) + \\
& \qquad e^{|\sigma_i|} \cdot d_{FS(Y)}\big(\Phi_{T+\tilde{\tau}_i} \circ \iota(b_i^{-1},\tilde{f}_i(c_i)),\Phi_{T} \circ \iota(e,c_i)\big) \\
& \leq e^{\alpha} \cdot \frac{\epsilon}{2n \cdot e^\alpha} + e^{\alpha} \cdot \frac{\epsilon}{2n  \cdot e^\alpha} = \frac{\epsilon}{n}.
\end{align*}
For $\tau := \sigma_n$ we obtain
\begin{align*}
& d_{FS(Y)}\big(\Phi_T \circ \iota(g,x),\Phi_{T+\tau} \circ \iota(h,y)\big) \\
& \leq \sum_{i=1}^n d_{FS(Y)}\big(\Phi_{T+\sigma_{i-1}} \circ \iota(g_{i-1},c_{i-1}),\Phi_{T+\sigma_i} \circ \iota(g_i,c_i)\big) \\
& \leq \sum_{i=1}^n \frac{\epsilon}{n} = \epsilon.
\end{align*}
This finishes the proof of Lemma~\ref{lem-cat0}.
\end{proof}

The proof of Theorem~\ref{thm-main} is based on the following proposition which is a modification of \cite[Proposition 3.9]{BL09}.
\begin{proposition} \label{prop-str}
Let $G$ be a group which is strongly transfer reducible over a family $\mathcal{F}$ of subgroups. Let $N$ be the number appearing in the definition of "strongly transfer reducible". Let $S \subseteq G$ be a finite symmetric subset containing the trivial element $e \in G$. Then for every $k \in \N$ there exist
\begin{itemize}
\item a compact contractible controlled $N$-dominated metric space $X$,
\item a strong homotopy $G$-action $\Psi$ on $X$,
\item a positive real number $\Lambda$,
\item a simplicial complex $\Sigma$ of dimension $\leq N$ with a simplicial cell preserving $G$-action and
\item a $G$-equivariant map $f \colon G \times X \to \Sigma$
\end{itemize}
such that
\begin{itemize}
\item the isotropy groups of $\Sigma$ belong to $\mathcal{F}$ and
\item $k \cdot d^1\big(f(g,x),f(h,y)\big) \leq d_{\Psi,S,k,\Lambda}\big((g,x),(h,y)\big)$ for all $(g,x),(h,y) \in G \times X$.
\end{itemize}
Here $d^1$ denotes the $l^1$-metric on simplicial complexes, see \cite[subsection 4.2]{BLR08}.
\end{proposition}
\begin{proof}
We choose a strong homotopy $G$-action $\Psi$ on a metric space $X$ and a cover $\mathcal{U}$ of $G \times X$ which satisfy the properties stated in Definition~\ref{def-str} for $S$, $k$ and $n := 4Nk$.
Using Lemma~\ref{lem-qm}~(\ref{lem-qm2}) we conclude as in the proof of \cite[Proposition 3.7]{BL09} that for every $x \in X$ there exists $\Lambda_x > 0$ and $U_x \in \mathcal{U}$ such that the $n$-ball around $(e,x)$ with respect to the quasi-metric $d_{\Psi,S,k,\Lambda_x}$ lies in $U_x$. Moreover, since $X$ is compact, there exists $\Lambda > 0$ such that every $n$-ball with respect to the quasi-metric $d_{\Psi,S,k,\Lambda}$ lies in some $U \in \mathcal{U}$ (see the proof of \cite[Proposition 3.7]{BL09}).
Let $\Sigma := |\mathcal{U}|$ be the realization of the nerve of $\mathcal{U}$ and let $f$ be the map induced by $\mathcal{U}$, i.e.
\[
f \colon G \times X \to |\mathcal{U}|, \, (g,x) \mapsto \sum_{U \in \mathcal{U}} \frac{a_U(g,x)}{s(g,x)} \, U
\]
with $a_U(g,x) := \inf\{d_{\Psi,S,k,\Lambda}((g,x),(h,y)) \, | \, (h,y) \notin U\}$ and $s(g,x) := \sum_U a_U(g,x)$.
Notice that $|a_U(g,x) - a_U(h,y)| \leq d_{\Psi,S,k,\Lambda}((g,x),(h,y))$ and hence
\[
\sum_{U \in \mathcal{U}} |a_U(g,x)-a_U(h,y)| \leq 2N \cdot d_{\Psi,S,k,\Lambda}\big((g,x),(h,y)\big)
\]
for all $(g,x), (h,y) \in G \times X$.
We calculate
\begin{eqnarray*}
d^1(f(g,x),f(h,y)) & = & \sum_{U \in \mathcal{U}} \Big| \frac{a_U(g,x)}{s(g,x)} - \frac{a_U(h,y)}{s(h,y)} \Big| \\
& = & \sum_{U \in \mathcal{U}} \Big| \frac{a_U(g,x)-a_U(h,y)}{s(g,x)} + \frac{a_U(h,y) \cdot \big(s(h,y)-s(g,x)\big)}{s(g,x) \cdot s(h,y)} \Big| \\
& \leq & \frac{\sum_{U \in \mathcal{U}} |a_U(g,x)-a_U(h,y)|}{s(g,x)} + \frac{\big|s(h,y)-s(g,x)\big|}{s(g,x)} \\
& \leq & 2 \cdot \frac{\sum_{U \in \mathcal{U}} |a_U(g,x)-a_U(h,y)|}{s(g,x)} \\
& \leq & 4N \cdot \frac{d_{\Psi,S,k,\Lambda}\big((g,x),(h,y)\big)}{s(g,x)} \\
& \leq & \frac{d_{\Psi,S,k,\Lambda}\big((g,x),(h,y)\big)}{k}.
\end{eqnarray*}
For the last inequality we used the fact that the $n$-ball around $(g,x)$ lies in some $U \in \mathcal{U}$ and hence $s(g,x) \geq n$.
\end{proof}

\section{The obstruction category} \label{sec-obscat}

In this section we recall the definition of the obstruction category. In the following $\mathcal A$ denotes a small additive category (with strictly associative direct sum) which is provided with a strict right $G$-action.

\begin{definition}
Let $X$ be a $G$-space and let $(Y,d_Y)$ be a metric space with an isometric $G$-action. We consider the $G$-space $G \times X \times Y \times [1,\infty)$ with the $G$-action given by $h(g,x,y,t) := (hg,hx,hy,t)$. We define the \emph{obstruction category} $\mathcal{O}^G(X,(Y,d_Y);\mathcal{A})$ as follows.\\
An object in $\mathcal{O}^G(X,(Y,d_Y);\mathcal{A})$ is a collection $A = (A_{g,x,y,t})_{(g,x,y,t) \in G \times X \times Y \times [1,\infty)}$ of objects in $\mathcal{A}$ with the following properties:
\begin{itemize}
 \item $A$ is locally finite, i.e. for every $z_0 \in G \times X \times Y \times [1,\infty)$ there exists an open neighborhood $U$ such that the set $\{z \in G \times X \times Y \times [1,\infty) \, | \, A_z \neq 0\} \cap U$ is finite.
 \item There is a compact subset $K \subseteq G \times X \times Y$ such that $A_{g,x,y,t} = 0$ whenever $(g,x,y) \notin G \cdot K$.
 \item We have $A_z \cdot g = A_{g^{-1}z}$ for all $z \in G \times X \times Y \times [1,\infty)$ and $g \in G$.
\end{itemize}
A morphism $\phi \colon B \to A$ is a collection of morphisms $\phi_{z,z'} \colon B_{z'} \to A_z$ in $\mathcal{A}$ ($z,z' \in G \times X \times Y \times [1,\infty)$) with the following properties:
\begin{itemize}
 \item The sets $\{ z \in G \times X \times Y \times [1,\infty) \, | \, \phi_{z,z'} \neq 0\}$ and $\{ z \in G \times X \times Y \times [1,\infty)  \, | \, \phi_{z',z} \neq 0\}$ are finite for all $z' \in G \times X \times Y \times [1,\infty)$.
 \item There are $R,T > 0$ and a finite subset $F \subseteq G$ such that $\phi_{(g,x,y,t),(g',x',y',t')} = 0$ whenever $g^{-1}g' \notin F$ or $d_Y(y,y') > R$ or $|t-t'| > T$.
 \item The set
     \[
      \big\{ \big((x,t),(x',t')\big) \in (X \times [1,\infty))^2 \, \big| \, \exists \, g,g' \in G, y,y' \in Y: \phi_{(g,x,y,t),(g',x',y',t')} \neq 0 \big\}
     \]
     lies in the equivariant continuous control condition $\mathcal{E}^X_{Gcc}$ defined in \cite[section 3.2]{BLR08}.
 \item We have $\phi_{z,z'} \cdot g = \phi_{g^{-1}z,g^{-1}z'}$ for all $z,z' \in G \times X \times Y \times [1,\infty)$ and $g \in G$.
\end{itemize}
Composition is given by matrix multiplication, i.e.
\[
 (\psi \circ \phi)_{z,z''} := \sum_{z' \in G \times X \times Y \times [1,\infty)} \psi_{z,z'} \circ \psi_{z',z''}.
\]
The obstruction category $\mathcal{O}^G(X,(Y,d_Y);\mathcal{A})$ inherits the structure of an additive category from $\mathcal{A}$.
\end{definition}
We use the same notation as in \cite[subsection 4.4]{BL09} which slightly differs from the notation used in \cite{BLR08} (see \cite[Remark 4.10]{BL09}).

The construction is functorial in $Y$: Let $f \colon Y \to Y'$ be a $G$-equivariant map with the property that for every $r > 0$ there exists $R > 0$ such that $d_{Y'}(f(y_1),f(y_2)) < R$ whenever $d_Y(y_1,y_2) < r$. Then the map $f$ induces a functor $f_* \colon \mathcal{O}^G(X,Y;\mathcal{A}) \to \mathcal{O}^G(X,Y';\mathcal{A})$ with $f_*(A)_{g,x,y',t} := \oplus_{y \in f^{-1}(\{y'\})} \, A_{g,x,y,t}$.

We are mostly interested in $\mathcal{O}^G(E_\mathcal{F}G,\pt;\mathcal{A})$ because of the following proposition which is proven in \cite[Proposition 3.8]{BLR08}.
\begin{proposition} \label{prop-assembly}
Let $G$ be a group and $m_0 \in \Z$ such that
\[
 K_m\big(\mathcal{O}^G(E_\mathcal{F}G,\pt;\mathcal{A})\big) = 0
\]
for all $m \geq m_0$ and all additive $G$-categories $\mathcal{A}$.
Then the assembly map~(\ref{eq-assembly}) is an isomorphism for all $m \in \Z$ and all additive $G$-categories $\mathcal{A}$.
\end{proposition}

The reason why we study the category $\mathcal{O}^G(E_\mathcal{F}G,(Y,d_Y);\mathcal{A})$ not only for $Y := \pt$ is that we need room for certain constructions.
Moreover, we want to consider simultaneously metric spaces $(Y_n,d_n)$ with isometric $G$-action ($n \in \N$). In analogy to \cite[subsection 3.4]{BLR08} we define the additive subcategory
\[
\mathcal{O}^G\big(E_\mathcal{F}G,(Y_n,d_n)_{n \in \N};\mathcal{A}\big) \subseteq \prod_{n \in \N} \mathcal{O}^G\big(E_\mathcal{F}G,(Y_n,d_n);\mathcal{A}\big)
\]
by requiring additional conditions on the morphisms. A morphism $\phi = (\phi(n))_{n \in \N}$ is allowed if there are $R > 0$ and a finite subset $F \subseteq G$ (not depending on $n$) such that $\phi(n)_{(g,x,y,t),(g',x',y',t')} = 0$ whenever $g^{-1}g' \notin F$ or $d_n(y,y') > R$.

The inclusion
\[
\bigoplus_{n \in \N} \mathcal{O}^G\big(E_\mathcal{F}G,(Y_n,d_n);\mathcal{A}\big) \to \mathcal{O}^G\big(E_\mathcal{F}G,(Y_n,d_n)_{n \in \N};\mathcal{A}\big)
\]
is a Karoubi filtration and we denote the quotient by $\mathcal{O}^G(E_\mathcal{F}G,(Y_n,d_n)_{n \in \N};\mathcal{A})^{> \oplus}$.
Notice that a sequence of $G$-equivariant maps $(f_n \colon Y_n \to Y'_n)_{n \in \N}$ induces a functor $(f_n)_* \colon \mathcal{O}^G(E_\mathcal{F}G,(Y_n,d_n)_{n \in \N};\mathcal{A}) \to \mathcal{O}^G(E_\mathcal{F}G,(Y'_n,d'_n)_{n \in \N};\mathcal{A})$ if for every $r > 0$ there exists $R > 0$ such that $d'_n(f_n(y_1),f_n(y_2)) < R$ whenever $d_n(y_1,y_2) < r$.

\section{Outline of the proof of Theorem~\ref{thm-main}} \label{sec-out}

In this section we sketch the proof of Theorem~\ref{thm-main}.

Since the class of groups satisfying the $K$-theoretic Farrell-Jones conjecture is closed under directed colimits, it suffices to prove the bijectivity of the K-theoretic assembly map~(\ref{eq-assembly}) for every finitely generated subgroup $H$ of $G$ (with respect to the family $\mathcal{F}_H := \{ F \cap H \mid F \in \mathcal{F} \}$). Moreover, strong transfer reducibility is stable under taking subgroups (see Remark~\ref{rem-subgr}). This shows that it is enough to prove Theorem~\ref{thm-main} for finitely generated groups. Therefore we can and will assume that $G$ is finitely generated.

We fix a finite symmetric generating subset $S \subseteq G$ which contains the trivial element $e \in G$. We apply Proposition~\ref{prop-str} to $S^n := \{ s_1 \cdot s_2 \cdot \ldots \cdot s_n \, | \, s_i \in S \} \subseteq G$ and $k := n$ and obtain
\begin{itemize}
\item compact contractible controlled $N$-dominated metric spaces $X_n$,
\item strong homotopy $G$-actions $\Psi_n$ on $X_n$,
\item positive real numbers $\Lambda_n$,
\item simplicial complexes $\Sigma_n$ of dimension $\leq N$ with simplicial cell preserving $G$-actions and
\item $G$-equivariant maps $f_n \colon G \times X_n \to \Sigma_n$
\end{itemize}
such that
\begin{itemize}
\item the isotropy groups of $\Sigma_n$ belong to $\mathcal{F}$ and
\item $n \cdot d^1(f_n(g,x),f_n(h,y)) \leq d_{\Psi_n,S^n,n,\Lambda_n}((g,x),(h,y))$ for all $(g,x),(h,y) \in G \times X_n$.
\end{itemize}
We abbreviate $d_n := d_{\Psi_n,S^n,n,\Lambda_n}$.

By Proposition~\ref{prop-assembly} it suffices to show
\[
 K_m\big(\mathcal{O}^G(E_\mathcal{F}G,\pt;\mathcal{A})\big) = 0
\]
for all $m \geq 1$.
Theorem~\ref{thm-main} is a consequence of the following commuting diagram:
\begin{equation*} \label{diagram}
\xymatrix{
K_m\big(\mathcal{O}^G(E_\mathcal{F}G,\pt;\mathcal{A})\big) \ar@{^{(}->}[d]^{\diag_*} \ar[rd]^{\trans_*} & \\
K_m\big(\mathcal{O}^G(E_\mathcal{F}G,(\pt)_{n \in \N};\mathcal{A})^{> \oplus}\big) \ar[d]^{=} & \ar[l]_{\hspace{-5mm} \pr_*} K_m\big(\mathcal{O}^G(E_\mathcal{F}G,(G \times X_n,d_n)_{n \in \N};\mathcal{A})^{> \oplus}\big) \ar[d]^{(f_n)_*} \\
K_m\big(\mathcal{O}^G(E_\mathcal{F}G,(\pt)_{n \in \N};\mathcal{A})^{> \oplus}\big) & \ar[l]_{\hspace{-5mm} \pr_*} K_m\big(\mathcal{O}^G(E_\mathcal{F}G,(\Sigma_n,n \cdot d^1)_{n \in \N};\mathcal{A})^{> \oplus}\big) = 0
}
\end{equation*}
The map $\diag_* \colon K_m(\mathcal{O}^G(E_\mathcal{F}G,\pt;\mathcal{A})) \to K_m(\mathcal{O}^G(E_\mathcal{F}G,(\pt)_{n \in \N};\mathcal{A})^{> \oplus})$ is induced by the diagonal map. The injectivity of this map can easily be shown by a diagram chase in the diagram
\begin{footnotesize}
\begin{equation*}
\xymatrix{
& K_m\big(\mathcal{O}^G(E_\mathcal{F}G,\pt;\mathcal{A})\big) \ar[d]^{\diag_*} \ar[rd]^{\diag_*} & \\
K_m\big(\displaystyle\bigoplus_{n \in \N} \mathcal{O}^G(E_\mathcal{F}G,\pt;\mathcal{A})\big) \ar[r] \ar[d]^{\bigoplus_{n \in \N} \pr_n}_\cong & K_m\big(\mathcal{O}^G(E_\mathcal{F}G,(\pt)_{n \in \N};\mathcal{A})\big) \ar[d]^{\prod_{n \in \N} \pr_n} \ar[r] & K_m\big(\mathcal{O}^G(E_\mathcal{F}G,(\pt)_{n \in \N};\mathcal{A})^{> \oplus}\big) \\
\displaystyle\bigoplus_{n \in \N} K_m\big(\mathcal{O}^G(E_\mathcal{F}G,\pt;\mathcal{A})\big) \ar[r] & \displaystyle\prod_{n \in \N} K_m\big(\mathcal{O}^G(E_\mathcal{F}G,\pt;\mathcal{A})\big) &
}
\end{equation*}
\end{footnotesize}
where the middle row comes from the Karoubi filtration. (Notice that the composition $\prod_{n \in \N} \pr_n \circ \diag_*$ in the middle column is the diagonal map.)
The transfer map
\[
\trans_* \colon K_m\big(\mathcal{O}^G(E_\mathcal{F}G,\pt;\mathcal{A})\big) \to K_m\big(\mathcal{O}^G(E_\mathcal{F}G,(G \times X_n,d_n)_{n \in \N};\mathcal{A})^{> \oplus}\big)
\]
will be constructed in section~\ref{sec-trans}.
The maps $\pr_*$ are induced by the projections $\pr \colon G \times X_n \to \pt$ resp. $\pr \colon \Sigma_n \to \pt$.
The equation
\[
K_m\big(\mathcal{O}^G(E_\mathcal{F}G,(\Sigma_n,n \cdot d^1)_{n \in \N};\mathcal{A})^{> \oplus}\big) = 0
\]
is proved in \cite[Theorem 7.2]{BLR08}.

\section{Preparations for the transfer} \label{sec-prep}

We will define the transfer map
\[
\trans_* \colon K_m\big(\mathcal{O}^G(E_\mathcal{F}G,\pt;\mathcal{A})\big) \to K_m\big(\mathcal{O}^G(E_\mathcal{F}G,(G \times X_n,d_n)_{n \in \N};\mathcal{A})^{> \oplus}\big).
\]
as the map induced by a functor
\[
\trans \colon \mathcal{O}^G(E_\mathcal{F}G,\pt;\mathcal{A}) \to \widetilde{\ch}_\hfd \mathcal{O}^G(E_\mathcal{F}G,(G \times X_n,d_n)_{n \in \N};\mathcal{A})^{> \oplus}.
\]

In this section we give a quite short review of the construction of the category $\widetilde{\ch}_\hfd \mathcal{O}^G(E_\mathcal{F}G,(G \times X_n,d_n)_{n \in \N};\mathcal{A})^{> \oplus}$. For more details we refer to \cite[subsection 6.2]{BLR08}.

For a metric space $(Y,d_Y)$ with an isometric $G$-action we define the category $\overline{\mathcal{O}}^G(E_\mathcal{F}G,(Y,d_Y);\mathcal{A}^\kappa)$ in the same way as in section~\ref{sec-obscat} but we replace $\mathcal{A}$ by $\mathcal{A}^\kappa$ for a fixed (suitably chosen) infinite cardinal $\kappa$ and drop the assumption that the support of objects is locally finite. Moreover, instead of requiring for a morphism $\phi = (\phi_{z,z'})_{z,z' \in G \times E_\mathcal{F}G \times Y \times [1,\infty)}$ that the sets $\{ z \, | \, \phi_{z,z'} \neq 0\}$ and $\{ z \, | \, \phi_{z',z} \neq 0\}$ are finite, we define a morphism
\[
 \phi = (\phi_{z,z'}) \colon B = (B_{z'})_{z' \in G \times E_\mathcal{F}G \times Y \times [1,\infty)} \to A = (A_z)_{z \in G \times E_\mathcal{F}G \times Y \times [1,\infty)}
\]
to be a morphism $\bigoplus_{z' \in G \times E_\mathcal{F}G \times Y \times [1,\infty)} B_{z'} \to \bigoplus_{z \in G \times E_\mathcal{F}G \times Y \times [1,\infty)} A_z$ in the category ${\mathcal A}^\kappa$.
For a sequence $(Y_n,d_n)_{n \in \N}$ of metric spaces with isometric $G$-action we define
\[
\overline{\mathcal{O}}^G(E_\mathcal{F}G,(Y_n,d_n)_{n \in \N};\mathcal{A}^\kappa) \subset \prod_{n \in \N} \overline{\mathcal{O}}^G(E_\mathcal{F}G,(Y_n,d_n);\mathcal{A}^\kappa)
\]
by requiring additional conditions on the morphisms precisely as in section~\ref{sec-obscat}. The inclusion
\[
\bigoplus_{n \in \N} \overline{\mathcal{O}}^G(E_\mathcal{F}G,(Y_n,d_n);\mathcal{A}^\kappa) \to \overline{\mathcal{O}}^G(E_\mathcal{F}G,(Y_n,d_n)_{n \in \N};\mathcal{A}^\kappa)
\]
is a Karoubi filtration and we denote the quotient by $\overline{\mathcal{O}}^G(E_\mathcal{F}G,(Y_n,d_n)_{n \in \N};\mathcal{A}^\kappa)^{> \oplus}$.
For the rest of this section we abbreviate
\begin{eqnarray*}
 \mathcal{O} & := & \mathcal{O}^G(E_\mathcal{F}G,(Y_n,d_n)_{n \in \N};\mathcal{A})^{> \oplus}, \\
 \overline{\mathcal{O}} & := & \overline{\mathcal{O}}^G(E_\mathcal{F}G,(Y_n,d_n)_{n \in \N};\mathcal{A}^\kappa)^{> \oplus}.
\end{eqnarray*}
One should think of the inclusion $\mathcal{O} \subset \overline{\mathcal{O}}$ as an inclusion of a full additive subcategory on objects satisfying finiteness conditions into a large category which gives room for constructions.

Let $\mathcal{C}$ be an additive category (e.g. $\mathcal{O}$ or $\overline{\mathcal{O}}$). We write $\Idem(\mathcal{C})$ for its idempotent completion. We define $\ch_\f(\mathcal{C})$ to be the category of chain complexes in $\mathcal{C}$ that are bounded above and below and $\ch^{\geq}(\mathcal{C})$ to be the category of chain complexes that are bounded below.
We write $\ch_\hf(\Idem(\mathcal{O}) \subset \Idem(\overline{\mathcal{O}}))$ for the full subcategory of $\ch^{\geq}(\Idem(\overline{\mathcal{O}}))$ consisting of chain complexes which are chain homotopy equivalent to a chain complex in $\ch_\f(\Idem(\mathcal{O}))$. We write $\ch_\hfd(\mathcal{O})$ for the full subcategory of $\ch^{\geq} \Idem(\overline{\mathcal{O}})$ consisting of objects $C$ which are homotopy retracts of objects in $\ch_\f(\mathcal{O})$, i.e. there exists a diagram $C \xrightarrow{i} D \xrightarrow{r} C$ with $D \in \ch_\f(\mathcal{O})$ such that the composition $r \circ i$ is chain homotopic to the identity on $C$.

The category $\ch_\hfd(\mathcal{O})$ is a Waldhausen category: The notion of chain homotopy leads to a notion of weak equivalence, and we define cofibrations to be those chain maps which are degree-wise the inclusion of a direct summand.
The following lemma is proven in \cite[Lemma 6.5]{BLR08}.
\begin{lemma}
The inclusion $\mathcal{O} \subset \ch_\hfd(\mathcal{O})$ induces an equivalence on $K_m$ for all $m \geq 1$.
\end{lemma}

We recall from \cite[subsection 8.2]{BR05} that for a given Waldhausen category $\mathcal{W}$ there exists a Waldhausen category $\widetilde{\mathcal W}$ whose objects are sequences
\[
 C_0 \xrightarrow{c_0} C_1 \xrightarrow{c_1} C_2 \xrightarrow{c_2} \cdots
\]
where the $c_\alpha$ are morphisms in ${\mathcal W}$ that are both cofibrations and weak equivalences. A morphism $f$ in $\widetilde{\mathcal W}$ is represented by a sequence of morphisms $(f_\alpha, f_{\alpha+1}, f_{\alpha+2}, \cdots)$ which makes the diagram
\begin{equation*}
\xymatrix{
C_\alpha \ar[d]^{f_\alpha} \ar[r]^{c_\alpha} & C_{\alpha+1} \ar[d]^{f_{\alpha+1}} \ar[r]^{c_{\alpha+1}} & C_{\alpha+2} \ar[d]^{f_{\alpha+2}} \ar[r]^{c_{\alpha+2}} & \cdots \\
D_{\alpha+k} \ar[r]^{d_{\alpha+k}} & D_{\alpha+k+1} \ar[r]^{d_{\alpha+k+1}} & D_{\alpha+k+2} \ar[r]^{d_{\alpha+k+2}} & \cdots
}
\end{equation*}
($\alpha,k \in \N_0$) commutative. If we enlarge $\alpha$ or $k$ the resulting diagrams represent the same morphism, i.e. we identify $(f_\alpha, f_{\alpha+1}, f_{\alpha+2}, \cdots)$ with $(f_{\alpha+1}, f_{\alpha+2}, f_{\alpha+3}, \cdots)$ but also with $(d_{\alpha+k} \circ f_\alpha, d_{\alpha+k+1} \circ f_{\alpha+1}, d_{\alpha+k+2} \circ f_{\alpha+2}, \cdots)$. Sending an object to the constant sequence defines an inclusion ${\mathcal W} \to {\widetilde {\mathcal W}}$. According to \cite[Proposition 8.2]{BR05} the inclusion induces an isomorphism on $K_m$ for $m \geq 0$ under some mild conditions for ${\mathcal W}$. These conditions will be satisfied in all our examples.

\section{The transfer} \label{sec-trans}

In this section we define a functor
\[
\trans \colon \mathcal{O}^G(E_\mathcal{F}G,\pt;\mathcal{A}) \to \widetilde{\ch}_\hfd (\mathcal{O}^G(E_\mathcal{F}G,(G \times X_n,d_n)_{n \in \N};\mathcal{A})^{> \oplus})
\]
which induces the desired transfer map
\[
\trans_* \colon K_m\big(\mathcal{O}^G(E_\mathcal{F}G,\pt;\mathcal{A})\big) \to K_m\big(\mathcal{O}^G(E_\mathcal{F}G,(G \times X_n,d_n)_{n \in \N};\mathcal{A})^{> \oplus}\big).
\]

The strong homotopy $G$-actions $\Psi_n$ on $X_n$ induce $G$-spaces $M_n$ (see Remark~\ref{rem_cha}). Moreover, we obtain filtrations
\[
 M_n^0 \subseteq M_n^1 \subseteq M_n^2 \subseteq \cdots \subseteq \bigcup_{\alpha \in \N_0} M_n^\alpha = M_n
\]
with
\[
 M_n^\alpha := \big\{ \Psi_n(?,t_\alpha,g_{\alpha-1},\ldots,g_0,x) \, \big| \, t_i \in [0,1], g_i \in S^\alpha, x \in X_n \big\}.
\]

We define maps $i_n^\alpha \colon M_n^\alpha \to X_n$ and $p_n \colon X_n \to M_n^0 \subseteq M_n^\alpha$ by $i_n^\alpha(f) := f(e)$ and $p_n(x) := \Psi_n(?,x)$.
$H_n^\alpha \colon M_n^\alpha \times [0,1] \to M_n^\alpha, (f,t) \mapsto f(?,t,e)$ is a homotopy between $H_n^\alpha(?,0) = p_n \circ i_n^\alpha$ and $H_n^\alpha(?,1) = \id$.

We denote by $C_*(n,\alpha) \subseteq C_*^\sing(G \times X_n)$ resp. $D_*(n,\alpha) \subseteq C_*^\sing(G \times M_n^\alpha)$ the chain subcomplex generated by all singular simplices $\sigma \colon \Delta \to G \times X_n$ resp. $G \times M_n^\alpha$ for which the diameter of $\sigma(\Delta)$ resp. $(\id \times i_n^\alpha) \circ \sigma(\Delta)$ is less or equal to $2\alpha$. $C_*(n,\alpha)$ and $D_*(n,\alpha)$ are both chain complexes over $G \times X_n$ via the barycenter map resp. the composition of the barycenter map with $\id \times i_n^\alpha$.

Let $A$ be an object in $\mathcal{O}^G(E_\mathcal{F}G,\pt;\mathcal{A})$. We define objects $(A \otimes C_*(n,\alpha))_{n \in \N}$ and $(A \otimes D_*(n,\alpha))_{n \in \N}$ in $\ch^{\geq} \overline{\mathcal{O}}^G(E_\mathcal{F}G,(G \times X_n,d_n)_{n \in \N};\mathcal{A^\kappa})$ by
\begin{eqnarray*}
\big(A \otimes C_k(n,\alpha)\big)_{(g,e,(h,x),t)} & := & \left\{ \begin{array}{ll} A_{(g,e,t)} \otimes C_k(n,\alpha)_{(g,x)} & \mbox{if $g=h$} \\ 0 & \mbox{otherwise} \end{array} \right., \\
\big(A \otimes D_k(n,\alpha)\big)_{(g,e,(h,x),t)} & := & \left\{ \begin{array}{ll} A_{(g,e,t)} \otimes D_k(n,\alpha)_{(g,x)} & \mbox{if $g=h$} \\ 0 & \mbox{otherwise} \end{array} \right..
\end{eqnarray*}
The differentials are given by $\id \otimes \partial$.

Furthermore, \cite[Lemma 8.4]{BL09} implies that $(A \otimes C_*(n,\alpha))_{n \in \N}$ is an object in $\ch_\hfd(\mathcal{O}^G(E_\mathcal{F}G,(G \times X_n,d_n)_{n \in \N};\mathcal{A})^{> \oplus})$. The usual construction of the chain homotopy associated to the homotopy $\id \times H_n^\alpha \colon G \times M_n^\alpha \to G \times M_n^\alpha$ yields a chain homotopy between $(\id \times (p_n \circ i_n^\alpha))_* \colon C_*(n,\alpha) \to C_*(n,\alpha)$ and the identity. Let $\sigma \colon \Delta \to G \times M_n^\alpha$ be a singular simplex for which the diameter of $i_n^\alpha \circ \pr_{M_n^\alpha} \circ \sigma(\Delta)$ is less or equal to $2\alpha$. Since $i_n^\alpha \circ H_n^\alpha(?,t) = i_n^\alpha$ for all $t$, the diameter of images of simplices in $\Delta \times [0,1]$ under $i_n^\alpha \circ \pr_{M_n^\alpha} \circ (\id \times H_n^\alpha) \circ (\sigma \times \id)$ is again bounded by $2\alpha$. This shows that $(A \otimes D_*(n,\alpha))_{n \in \N}$ is a homotopy retract of $(A \otimes C_*(n,\alpha))_{n \in \N}$ and that $(A \otimes D_*(n,\alpha))_{n \in \N}$ is an object in $\ch_\hfd(\mathcal{O}^G(E_\mathcal{F}G,(G \times X_n,d_n)_{n \in \N};\mathcal{A})^{> \oplus})$, too.

We define $\trans(A)$ as the object
\[
 \big(A \otimes D_*(n,0)\big)_{n \in \N} \xrightarrow{\id \otimes \inc} \big(A \otimes D_*(n,1)\big)_{n \in \N} \xrightarrow{\id \otimes \inc} \big(A \otimes D_*(n,2)\big)_{n \in \N} \xrightarrow{\id \otimes \inc} \ldots
\]
in $\widetilde{\ch}_\hfd \mathcal{O}^G(E_\mathcal{F}G,(G \times X_n,d_n)_{n \in \N};\mathcal{A})^{> \oplus}$.

Let $\phi \colon A \to B$ be a morphism in $\mathcal{O}^G(E_\mathcal{F}G,\pt;\mathcal{A})$. We choose $\alpha_0 \in \N$ such that $\phi_{(g,e,t),(g',e',t')} = 0$ whenever $g^{-1}g' \notin S^{\alpha_0}$.
For $\alpha \geq \alpha_0$ we define
\[
 \big(\phi \otimes m(n,\alpha)\big)_{n \in \N} \colon \big(A \otimes D_*(n,\alpha)\big)_{n \in \N} \to \big(B \otimes D_*(n,\alpha+1)\big)_{n \in \N}
\]
whose components are given by
\[
 \big(\phi \otimes m(n,\alpha)\big)_{(g,e,(g,x),t),(g',e',(g',x'),t')} = \phi_{(g,e,t),(g',e',t')} \otimes m_{g^{-1}g'}(n,\alpha)_{(g,x),(g',x')}.
\]
Here, $m_h(n,\alpha) \colon D_*(n,\alpha) \to D_*(n,\alpha+1)$ ($h \in S^{\alpha_0} \subseteq S^\alpha$) is the map induced by
\[
 G \times M_n^\alpha \to G \times M_n^{\alpha+1}, (g,f) \mapsto (gh^{-1},c_h(f))
\]
where $c_h \colon M_n^\alpha \to M_n^{\alpha+1}, f \mapsto f(?,1,h)$ is the restriction of the $G$-action on $M_n$.

Then $(\phi \otimes m(n,\alpha))_{n \in \N}$ is a morphism in $\ch_\hfd \mathcal{O}^G(E_\mathcal{F}G,(G \times X_n,d_n)_{n \in \N};\mathcal{A})^{> \oplus}$.
The crucial point is that we have
\[
 \big(\phi \otimes m(n,\alpha)\big)_{(g,e,(g,x),t),(g',e',(g',x'),t')} = 0
\]
for $n \geq \alpha \geq \alpha_0$ whenever $d_n((g,x),(g',x')) > 2$.
We will prove this fact. Suppose that $(\phi \otimes m(n,\alpha))_{(g,e,(g,x),t),(g',e',(g',x'),t')} \neq 0$ with $n \geq \alpha \geq \alpha_0$. We want to show $d_n((g,x),(g',x')) \leq 2$. We have $\phi_{(g,e,t),(g',e',t')} \neq 0$ and $m_{g^{-1}g'}(n,\alpha)_{(g,x),(g',x')} \neq 0$.
The first inequality shows $g^{-1}g' \in S^{\alpha_0} \subseteq S^\alpha$. The second inequality implies the existence of an element $f \in M_n^\alpha$ such that $f(e)=x'$ and $c_{g^{-1}g'}(f)(e)=x$. We write $f = \Psi_n(?,t_\alpha,g_{\alpha-1},\ldots,g_0,y)$ with $t_i \in [0,1]$, $g_i \in S^\alpha$, $y \in X_n$. We obtain
\[
 (g',x') = \big(g',f(e)\big), (g,x) = \big(g,c_{g^{-1}g'}(f)(e)\big) \in S^1_{\Psi_n,S^\alpha,\alpha}(g'k,y)
\]
with $k := g_{\alpha-1} \cdot \ldots \cdot g_0$. This implies
\[
 d_n\big((g,x),(g',x')\big) \leq d_n\big((g,x),(g'k,y)\big) + d_n\big((g',x'),(g'k,y)\big) \leq 2.
\]

Finally, we obtain a functor
\[
\trans \colon \mathcal{O}^G(E_\mathcal{F}G,\pt;\mathcal{A}) \to \widetilde{\ch}_\hfd \mathcal{O}^G(E_\mathcal{F}G,(G \times X_n,d_n)_{n \in \N};\mathcal{A})^{> \oplus}
\]
which sends a morphism $\phi \colon A \to B$ to the morphism represented by
\begin{equation*}
\xymatrix{
(A \otimes D_*(n,\alpha_0))_{n \in \N} \ar[d]^{(\phi \otimes m(n,\alpha_0))_{n \in \N}} \ar[r]^{\id \otimes \inc} & (A \otimes D_*(n,\alpha_0+1))_{n \in \N} \ar[d]^{(\phi \otimes m(n,\alpha_0+1))_{n \in \N}} \ar[r]^{\hspace{15mm} \id \otimes \inc} & \cdots \\
(B \otimes D_*(n,\alpha_0+1))_{n \in \N} \ar[r]^{\id \otimes \inc} & (B \otimes D_*(n,\alpha_0+2))_{n \in \N} \ar[r]^{\hspace{15mm} \id \otimes \inc} & \cdots
}
\end{equation*}

See \cite[Lemma 6.16]{BLR08} for the proof that
\[
\pr_* \circ \trans_* = \diag_* \colon K_m\big(\mathcal{O}^G(E_\mathcal{F}G,\pt;\mathcal{A})\big) \to K_m\big(\mathcal{O}^G(E_\mathcal{F}G,(\pt)_{n \in \N};\mathcal{A})^{> \oplus}\big).
\]

\end{document}